\documentclass[draft]{amsart}
\usepackage{amssymb,latexsym,amscd, amsthm, epsfig,amsmath,amssymb}
\usepackage{color}
\usepackage{xypic}
\usepackage[all]{xy}
\usepackage[T1]{fontenc}
\usepackage[latin1]{inputenc}
\usepackage[all]{xy}
\usepackage{verbatim}
\usepackage{amsfonts,amsthm}
\usepackage[twoside=false,left=3.3cm,right=3.3cm,top=3cm,bottom=3cm]{geometry}
\usepackage{cancel}
\usepackage{graphicx}
\usepackage{hyperref}
\usepackage{ bbold }
\usepackage{thmtools}
\usepackage{thm-restate}
\usepackage{cleveref}
 \newtheorem{thm}{Theorem}[section]

 \newtheorem{prop}[thm]{Proposition}
 
 \newtheorem{choice}{Choice}
 
 \theoremstyle{definition}
 \newtheorem{dfnt}[thm]{Definition}
  \theoremstyle{digression}
 
 \theoremstyle{remark}
 \newtheorem{rem}[thm]{Remark}
 \theoremstyle{definition}

\newtheoremstyle{theorem}{\topsep}{\topsep}
{\slshape}{}{\bf}{{\normalfont.}}{.5em}{}

\newtheorem*{maintheorem*}{Main Theorem}

\usepackage{xcolor}

\def\blu#1{\textcolor[rgb]{0,0,1}{#1}}

\newcommand{\pb}{\ar@{}[dr]|(.2){\blu{\diagup}}}
\newcommand{\po}{\ar@{}[ul]|(.2){\blu{\diagup}}}

\def\pr{\ensuremath{\mathop{\textrm{\normalfont pr}}}}

\def\id{\ensuremath{\mathop{\textrm{\normalfont id}}}}
\def\Id{\ensuremath{\mathop{\textrm{\normalfont Id}}}}

\title[A looping-delooping adjunction for topological spaces]{A looping-delooping adjunction for topological spaces}

\author[M. Rovelli]{Martina Rovelli}
\address[M. Rovelli]{Department of Mathematics, EPF Lausanne, Switzerland}
\email{martina.rovelli@epfl.ch}
\begin{document}
\maketitle

\begin{abstract}
Every principal $G$-bundle over $X$ is classified up to equivalence by a homotopy class $X\rightarrow BG$, where $BG$ is the classifying space of $G$. On the other hand, for every nice topological space $X$ Milnor constructed a strict model of its loop space $\tilde\Omega X$, that is a group. Moreover the morphisms of topological groups $\tilde{\Omega}X\rightarrow G$ generate all the $G$-bundles over $X$ up to equivalence.

In this paper, we show that the relation between Milnor's loop space and the classifying space functor is, in a precise sense, an adjoint pair between based spaces and topological groups in a homotopical context.

This proof leads to a classification of principal bundles over a fixed space, that is dual to the classification of bundles with a fixed group. Such a result clarifies the deep relation that exists between the theory of bundles, the classifying space construction and the loop space, which are very important in topological $K$-theory, group cohomology, and homotopy theory.
\end{abstract}
\tableofcontents

\section*{Introduction}
Given a topological group $G$ and a CW-complex $X$, a classical result (\cite{dold}, \cite{steenrod}) states that every $G$-bundle $E$ over $X$ is classified up to equivalence by the homotopy class of a map $f:X\rightarrow BG$, so that $f^{\star}EG\simeq E$. Here, the universal bundle $EG$ is a contractible $G$-bundle over the classifying space $BG$. An explicit description of $EG$ can be found in \cite{milnor2}. The bijection
\begin{equation}
\label{1}
{}_X\mathcal Bun_G/_{\simeq}\cong\mathcal Top(X,BG)/_{\simeq}
\end{equation}
suggests that the classifying space construction $B:\mathcal Gp(\mathcal Top)\rightarrow\mathcal Top$ could behave like a right adjoint in a homotopical context. Moreover, $B$ is in fact a functor, and the bijection (\ref{1}) is natural in $X$ and $G$.

For a topological group $G$, the correspondence (\ref{1}) witnesses the fact that $BG$ classifies principal $G$-bundles. Dualising the picture, it is natural to ask whether there exists an analougous classification for principal bundles over a fixed base space $X$. The answer was partially given by Milnor in \cite{milnor1}. For any based topological space $X$ with a good cellular decomposition he constructed a principal bundle $\tilde PX$ over $X$ whose structure group is denoted $\tilde\Omega X$. Such a bundle generates all the other principal bundles over $X$ up to equivalence, giving rise to a surjection of the form
\begin{equation}
\label{2}
\mathcal Gp(\mathcal Top)(\tilde\Omega X,G)\twoheadrightarrow{}_X\mathcal Bun_G/_{\simeq}.
\end{equation}
This assignment is dual in a precise way to the assignment (\ref{1}).

The bundle $\tilde PX$ is a strictification of the path space of $X$, so that concatenation becomes associative. The corresponding loop space $\tilde\Omega X$ is a group, and the fibration $\tilde PX\rightarrow X$ is a principal bundle. Following the intuition that $\tilde\Omega X$ is a loop space, the classifying map $\tilde\Omega X\rightarrow G$ for a $G$-bundle over $X$ is the connecting map that appears in the (dual) Nomura-Puppe sequence.

In this paper, we refine the bijection (\ref{1}), showing that every $G$-bundle over $X$ is classified by a \emph{pointed} map $X\rightarrow BG$, and we compute the kernel of the surjection (\ref{2}), i.e., we determine when two morphisms $\tilde\Omega X\rightarrow G$ produce equivalent bundles. This leads to the notion of \emph{algebraic equivalence}, denoted $\equiv$, of morphisms of topological groups. Although there is no natural extension of $\tilde\Omega$ to a functor, we propose a definition of $\tilde\Omega$ on arrows. The assignment is pseudofunctorial, according to the relation of algebraic equivalence, while the bijection obtained by (\ref{2}) is natural with respect to $X$ and $G$. This is summed up in the Main Theorem, whose precise statement appears as Theorem \ref{adjunction}.
\begin{maintheorem*}
\emph{
For every nice connected based space $X$ and nice topological group $G$, there are natural bijections
$$\mathcal Gp(\mathcal Top)(\tilde\Omega X,G)/_{\equiv}\cong{}_X\mathcal Bun_G/_{\simeq}\cong\mathcal Top_{\star}(X,BG)/_{\simeq}.$$
In this sense Milnor's loop space $\tilde\Omega$ and the classifying space functor $B$ form an adjoint pair up to homotopy.
}
\end{maintheorem*}
The result describes an intrinsic duality that relates loop spaces, classifying spaces and principal bundles, and gives a complete classification of principal bundles over a fixed space.

The condition of niceness essentially requires a countable CW-decomposition that is compatible with the base point or with the group structure. This restriction is mostly imposed by the construction $\tilde P$. In any case the conditions are relaxed enough to include interesting examples (e.g., all countable discrete groups, $\mathbb R$, and $\mathbb S^1$ as topological groups, and connected countable CW-complexes as based spaces).

Many of the ideas in this paper were inspired by \cite{fh}, where Farjoun and Hess developed the theory of \emph{twisted homotopical categories}. This is a formal framework based on a \emph{looping-delooping} adjunction
$$\Omega:\xymatrix{co\mathcal Mon\ar@<1ex>[r]&\mathcal Mon\ar@<1ex>[l]}:B$$
between suitable categories of monoids and comonoids, where a map $\Omega X\rightarrow G$ (or its adjoint $X\rightarrow BG$) describes a \emph{bundle}.
A $G$-bundle $E$ over $X$ is a twisted version of the tensor product between a comonoid $X$ and a monoid $G$, and always comes with a diagram of the form
$$\xymatrix{G\ar[r]&E\ar[r]&X},$$
that describes the projection on the base space and the inclusion of a distinguished fibre. Examples are given by twisted tensor products in chain complexes and twisted cartesian products in simplicial sets.
\\\ \\
\textbf{Related work.}
It is worth comparing the main result of this paper with the work of Lurie. Applying arguments of higher category theory and topos theory, he showed (\cite[Lemma 7.2.2.11]{lurie}) that there is an adjunction of quasi-categories between the $\infty$-topos of convenient pointed topological spaces and the $\infty$-topos of convenient topological groups. His formal framework allows a formulation of the result in term of a strict adjunction in the environment of quasi-categories.

Most likely realizing the $\pi_0$-statement in spaces would give a result close to ours, though Lurie's equivalence cannot be proved by using the argument in presented in this article. In fact, we already have trouble trying to see the loop functor as a simplicial map. We could try to assign to each pointed space (respectively pointed continuous map, respectively homotopy) a topological group (respectively a continuous homomorphism, respectively an equivalence of continuous homomorphisms). The problem is that the way we produce these correspondences often involves either a choice or the solution of a lifting problem. This issue cannot be overcome with the language of enriched categories either. On the other hand, the argument presented below uses more elementary tools, and it explains clearly how the looping-delooping adjunction is related to the theory of principal bundles. It also includes the description of the \emph{kernel} of the classification of bundles with a fixed structure group.

A second related article is \cite{lashof} by Lashof. He assigned to each equivalence class of $G$-bundles over $X$ a continuous groupoid morphism $\Omega X\longrightarrow G$ up to \emph{conjugacy equivalence}, in a bijective way. The groupoid structure defined on the loop space does not seem to be directly comparable to any of the known algebraic structures on the loop spaces. Also, the groupoid map is not obtained as a restriction of a map between the total objects. However the feeling is that the final classification presented in our paper is similar to his: in both cases a loop object is though of as a classifier.\\\ \\
\textbf{Future directions.}
We describe briefly here applications of the result in this article that we will elaborate in forthcoming articles.

It is possible to define a category $\mathcal Bun$ of bundles that includes fully faithfully pointed spaces and topological groups as degenerate cases.
Given an abelian nice topological group $H$, e.g., $H:=\mathbb Z,\mathbb Z_n,\mathbb S^1$, we can define an equivalence relation $\sim$ on $\mathcal Bun(E,H)$ for any bundle $E$. In particular, for any topological group $G$ and pointed space $X$ the quotients are
$$\mathcal Bun(G,H)/_{\sim}=\mathcal Gp(\mathcal Top)(G,H)/_{\equiv}\text{\quad and \quad}\mathcal Bun(X,H)/_{\sim}=\{\star\}.$$
Moreover, the multiplication of $H$ induces on $\mathcal Bun(E,H)/_{\sim}$ the structure of an abelian group.
We then define a second equivalence relation $\approx$ on $\mathcal Bun(E,BH)$. Since $BH$ is an abelian H-space, the quotient is an abelian group. For degenerate cases the relation gives now
$$\mathcal Bun(G,BH)/_{\approx}=\{\star\}\text{\quad and \quad}\mathcal Bun(X,BH)/_{\approx}=\mathcal Top_{\star}(X,BH)/_{\simeq}\cong\mathcal Gp(\mathcal Top)(\tilde\Omega X,H)/_{\equiv}.$$

Now, given any $G$-bundle $E$ over $X$, we can write a sequence in $\mathcal Bun$ using the two classifying maps of $E$:
$$\xymatrix{\tilde\Omega B\ar[r]&G\ar[r]&E\ar[r]&X\ar[r]&BG}.$$
By applying $\mathcal Bun(-,H)/_{\sim}$ and $\mathcal Bun(-,BH)/_{\approx}$ we obtain two sequences of groups, which can be connected by using the Main Theorem:
$$\tiny\xymatrix@C=1.3pc{\{\star\}\ar[r]&\{\star\}\ar[r]&\mathcal Bun(E,H)/_{\sim}\ar[r]&\mathcal Gp(\mathcal Top)(G,H)/_{\equiv}\ar[r]\ar[d]_-B&\mathcal Gp(\mathcal Top)(\tilde\Omega X,H)/_{\equiv}
\ar[d]_-{\cong}&&&\\
&&&\mathcal Top_{\star}(BG,BH)/_{\simeq}\ar[r]&\mathcal Top_{\star}(X,BH)/_{\simeq}\ar[r]&\mathcal Bun(E,BH)/_{\approx}\ar[r]&\{\star\}\ar[r]&\{\star\}.}
$$
There is work in progress to prove that the sequence of abelian groups
$$\xymatrix@C=1.2pc{\{\star\}\ar[r]&\mathcal Bun(E,H)/_{\sim}\ar[r]&\mathcal Gp(\mathcal Top)(G,H)/_{\equiv}\ar[r]&\mathcal Top_{\star}(X,BH)/_{\simeq}\ar[r]&\mathcal Bun(E,BH)/_{\approx}\ar[r]&\{\star\}}
$$
is exact.

When $H=K(A;n)$ is an Eilenberg-MacLane space, the quotient $\mathcal Gp(\mathcal Top)(G,H)/_{\equiv}$ is a subgroup of the group cohomology $H^{n+1}_{\mathcal Gp}(G;A)$ of $G$ with coefficients in $A$. We know examples when they actually coincide for small values of $n$, e.g., $H:=\mathbb Z,\mathbb Z_2,\mathbb S^1$. If this holds, the sequence can be written as
$$\xymatrix{0\ar[r]&\mathcal Bun(E,H)/_{\sim}\ar[r]&H^{n+1}_{\mathcal Gp}(G;A)\ar[r]&H^{n+1}(X;A)\ar[r]&\mathcal Bun(E,BH)/_{\approx}\ar[r]&0},$$
and gives a description of the kernel and the cokernel of the \emph{characteristic map} of $E$, which is used in the literature to define \emph{characteristic classes} of bundles (\cite{ms}).
The exactness of this sequence contains a lot of geometric information. For instance, when $H:=\mathbb Z_2$, $X$ is a real $k$-manifold and $E$ is its frame bundle, we recover the classical description of the first Stiefel-Whitney class as obstruction to orientability.

A dual formalism can be used to produce an exact sequence of groups, except for the last entry $\mathcal Bun(\Omega\mathbb S^n,E)/_{\approx}$, which is only a pointed set,
$$\xymatrix@C=1.2pc{\{\star\}\ar[r]&\mathcal Bun(\mathbb S^n,E)/_{\sim}\ar[r]&\pi_n(X)\ar[r]&\pi_{n-1}(G)\ar[r]&\mathcal Bun(\Omega\mathbb S^n,E)/_{\approx}\ar[r]&\{\star\}}
,$$
that describes the kernel and the cokernel of the connecting map in the \emph{homotopy long exact sequence} induced by $E$ (\cite[Theorem 3.1.5]{piccinini}).
\\\ \\
\textbf{Acknowledgements.}
I am grateful to my supervisor, Kathryn Hess, for her kind availability and her careful reading of this paper.
I would also like to thank Marc Stephan for many helpful discussions. Finally, I appreciate greatly the referee's constructive suggestions for the structure of this paper.
\section{Principal bundles}
\label{preliminaries}
In this section we recall the notion of principal bundle and two useful constructions: pullback and pushforward. They allow us to compare bundles with different structure groups or different base spaces.
\ \\
\textbf{Tensor product of modules.}
Let $G$ be a topological group, and let $\mathcal Mod_G$ and ${}_G\mathcal Mod$ denote the categories of right and left $G$ modules. Given $M$ a right $G$-module and $M'$ a left $G$-module, the \textbf{tensor product} $M\otimes_GM'$ is defined to be the coequalizer in $\mathcal Top$ of
$$\xymatrix{M\times G\times M'\ar@<1ex>[rr]^-{\mu\times M'}\ar[rr]_-{M\times \mu'}&&M\times M'}.$$
Thus, $M\otimes_GM'={M\times M'}/{\sim}$, where the relation is generated by the pairs $(m\cdot g,m')\sim(m,g\cdot m')$.
This construction extends to a bifunctor
$$-\otimes_G-:\mathcal Mod_{G}\times_G\mathcal Mod\longrightarrow\mathcal Top,$$
which is associative in the obvious sense.

Every morphism $a:G\rightarrow G'$ of topological groups induces an adjunction
$$a_{\star}:=-\otimes_GG':\xymatrix{\mathcal Mod_G\ar@<1ex>[r]&\mathcal Mod_{G'}\ar@<1ex>[l]}:U_a$$
between the \textbf{pushforward} $a_{\star}$ and the forgetful functor.
For any left $G$-module $M$, we denote the unit of this adjunction
$$\eta_a:M\longrightarrow a_{\star}M.$$
\ \\
\textbf{Cotensor product of comodules.}
Let $X$ be a topological space, and let $co\mathcal Mod_X$ and ${}_Xco\mathcal Mod$ denote the categories of right and left $X$-comodules. Note that, since we are working in a cartesian category, a right or left comodule $M$ over $X$ is equivalent to a map $\pi:M\rightarrow X$.

Given $M$ a (right) $X$-comodule and $M'$ a (left) $X$-comodule, the \textbf{cotensor product} $M\times_XM'$ is defined to be the equalizer in $\mathcal Top$ of
$$\xymatrix{M\times M'\ar@<1ex>[rr]^-{(M,\pi)\times M'}\ar[rr]_-{M\times(\pi',M')}&&M\times X\times M'}.$$
Thus, $M\times_XM'$ is just the pullback of $\pi$ and $\pi'$.
This construction extends to a bifunctor
$$-\times_X-:co\mathcal Mod_X\times{}_Xco\mathcal Mod\longrightarrow\mathcal Top$$
that is associative in the obvious sense.

Every continuous map $f:X'\rightarrow X$ of spaces induces an adjunction
$$U_f:\xymatrix{{}_{X'}co\mathcal Mod\ar@<1ex>[r]&{}_Xco\mathcal Mod\ar@<1ex>[l]}:X'\times_X-=:f^{\star}$$
between the forgetful functor and the \textbf{pullback} $f^{\star}$.
For any $X$-module $N$, we denote the counit
$$\epsilon_f:f^{\star}N\longrightarrow N.$$
\ \\
\textbf{Mixed modules.}
\label{mixedmodules}
Let $G$ be a topological group and $X$ a topological space. A \textbf{mixed module} over $X$ and $G$ is a topological space $M$ endowed with right $G$-module and (left) $X$-comodule structures such that the action is fibrewise, i.e., $\pi(m\cdot g)=\pi(m)$ for any $g\in G$ and $m\in M$.
A \textbf{morphism of mixed modules over $X$ and $G$} is a continuous map that is a $G$-equivariant morphism of spaces over $X$. An isomorphism of mixed modules over $X$ and $G$ is called an \textbf{equivalence}. We denote by $_X\mathcal Mix_G$ the category of mixed modules over $X$ and $G$. The \textbf{trivial mixed module} over $X$ and $G$ is $X\times G$.

For any continuous map $f:X'\rightarrow X$ and homomorphism of topological groups $a:G\rightarrow G'$, the adjunctions above restrict and corestrict to mixed modules and commute:
$$a_{\star}\circ f^{\star}\cong f^{\star}\circ a_{\star}:{}_X\mathcal Mix_{G}\longrightarrow{}_{X'}\mathcal Mix_{G'}.$$
\ \\
\textbf{Principal bundles.}
\label{principalbundle}
Let $G$ be a topological group and $X$ a topological space. A \textbf{principal bundle} with \textbf{structure group} $G$ and \textbf{base space} $X$, or a \textbf{$G$-bundle over $X$}, is a mixed module $P$ over $X$ and $G$ that is \textbf{locally trivial}, i.e., there exists an open covering $\{U_i\}_{i\in I}$ of $X$ and a \textbf{local trivialization} $\{\psi_i\}_{i\in I}$, where
$$\psi_i:P|_{U_i}:=\pi^{-1}(U_i)\longrightarrow U_i\times G$$
is an equivalence of mixed modules over $U_i$ and $G$. The trivial module $X\times G$ is a bundle. We denote by $_X\mathcal Bun_G$ the full subcategory of $_X\mathcal Mix_G$ of $G$-bundles over $X$. Every morphism of $G$-bundles over $X$ is in fact an equivalence.

Pullback and pushforward construction restrict and corestrict to principal bundles, i.e., they define assignments
$$(-)^{\star}(-):\mathcal T(X',X)\times{}_X\mathcal Bun_G\rightarrow{}_{X'}\mathcal Bun_G\text\quad{ and }\quad(-)_{\star}(-):\mathcal Gp(\mathcal Top)(G,G')\times{}_X\mathcal Bun_G\rightarrow{}_X\mathcal Bun_{G'}.$$
Using the universal property of pullbacks, it is easy to show that a $G$-equivariant morphism $P\rightarrow E$ between $G$-bundles that induces a morphism $f:X\rightarrow Y$ on the base spaces also induces an equivalence of $G$-bundles over $X$
$$P\simeq f^{\star}E.$$
Similarly, using the fact that the pushforward is a coequaliser, a morphism $Q\rightarrow E$ of bundles over $X$ that is equivariant according to a map $a:H\rightarrow G$ between the structure groups induces an equivalence of $H$-bundles over $X$
$$a_{\star}Q\simeq E.$$
\section{Classification of principal bundles with fixed structure group}
For a topological group $G$, a classification (stated here as Theorem \ref{classificationGbundles}) of $G$-bundles over CW-complexes was proven by Dold (in \cite{dold} for bundles over paracompact spaces, which include CW-complexes) and Steenrod (in \cite{steenrod} for bundles over normal, locally compact and countably compact spaces): every $G$-bundle over a CW-complex $X$ is determined up to equivalence by a homotopy class of maps from $X$ into the classifying space $BG$. The two references share the same ideas.

Exploiting the Serre model structure on spaces, we improve the classical result by proving a strong universal property of the classifying space of a group in Theorem \ref{generalizeduniversalpropertyclassifyingspace}. This implies the classification mentioned above.
On the other hand, it also allows a variant of the main classification where we take into consideration only continuous maps $X\rightarrow BG$ that are pointed (Theorem \ref{pointedclassificationGbundles}). Such a strong universal property was already proven by Steenrod in the case of bundles over locally finite complexes, as Theorem 19.4 of \cite{steenrod}.

\ \\
We think of the pullback construction as a tool to produce new principal bundles with the same structure group. Given a topological group $G$, one can fix a suitable $G$-bundle $Q$ over some space $X$, and let the map along which we change the base space vary. For a fixed $G$-bundle $Q$ over $Y$, this process is encoded by the following assignment:
$$(-)^{\star}Q:\xymatrix{\mathcal Top(X,Y)\ar[r]&_X\mathcal Bun_G},$$
where $[f:X\rightarrow Y]\mapsto f^{\star}Q$. This map is not surjective in a strict sense, but the goal is to capture all the $G$-bundles for a fixed group $G$ \emph{up to equivalence}. Therefore, the function we are really interested in is
$$(-)^{\star}Q:\xymatrix{\mathcal Top(X,Y)\ar[r]&_X\mathcal Bun_G\ar@{->>}[r]&_X\mathcal Bun_G/\simeq},$$
or its pointed version
$$(-)^{\star}Q:\xymatrix{\mathcal Top_{\star}(X,Y)\ar@{^{(}->}[r]&\mathcal Top(X,Y)\ar[r]&_X\mathcal Bun_G\ar@{->>}[r]&_X\mathcal Bun_G/\simeq},$$
for a fixed bundle $Q$.
There is a part of the kernel that does not depend on $Q$ and $Y$, which is a consequence of the \emph{Covering Homotopy Theorem} (\cite[Theorem 7.8]{dold}, or \cite[Theorem 11.3]{steenrod}).
\begin{prop}\cite[Corollary 7.10]{dold}, \cite[Theorem 11.5]{steenrod}.
\label{homotopicmapsgiveequivalentbundles}
For any $G$-bundle $Q$ over a CW-complex $Y$, its pullbacks via homotopic maps $X\rightarrow Y$ give equivalent $G$-bundles over $X$, i.e.,
$$f\simeq g:X\longrightarrow Y\quad\Longrightarrow\quad f^{\star}Q\simeq g^{\star}Q.$$
\end{prop}
So far, this is the best we can deduce without specializing to any particular bundle $Q$. We will now fix a group $G$ and focus on the \emph{universal bundle} of $G$, proposed by Milnor in \cite{milnor2}. He considers the join of infinitely many copies of $G$
$$\tilde EG:=\star_{n=0}^{+\infty}G,$$
which he endows with a right $G$-action induced by right multiplication on every copy of $G$. Motivated by its good properties, the orbit space $\tilde BG$ is often called \textbf{classifying space} of $G$. Moreover, $\tilde BG$ comes with a natural base point, and can be considered as a pointed space.

\begin{thm}\cite[Sections 3,5]{milnor2}
\label{classifyingspace}
The constructions $\tilde E$ and $\tilde B$ are functorial into spaces (or pointed spaces in the case of $\tilde B$). The space $\tilde EG$ is contractible, and has the structure of a $G$-bundle over $\tilde BG$. Moreover, if $G$ is a countable CW-group, then $\tilde BG$ is a countable CW-complex.
\end{thm}
The conditions for topological groups and for spaces appearing in the last part of the statement will play a role in Section 3, since \emph{Milnor's loop space} $\tilde\Omega$ will be defined only for spaces that admit a good CW-structure.
\begin{dfnt}[Nice spaces and nice groups]\ 
\begin{itemize}
\item A topological space $X$ is \textbf{nice} if it is connected and admits a countable CW-decomposition. We denote by $\mathcal T$ the category of nice pointed spaces and continuous maps.
\item A pointed topological space $X$ is \textbf{nice} if it is connected and admits a countable CW-decomposition such that the base point is a vertex. We denote by $\mathcal T_{\star}$ the category of nice pointed spaces and pointed continuous maps.
\item A topological group $G$ is \textbf{nice} if it admits the structure of a countable CW-group, i.e., a countable CW-decomposition with respect to which the group structure maps are cellular. We denote by $\mathcal G$ the category of of nice topological groups and topological groups homomorphisms.
\end{itemize}
\end{dfnt}
%
The following result is new in this generality. In fact, it generalizes Theorem 19.4 of \cite{steenrod}, where Steenrod proves it only for locally finite complexes, and it implies Theorem \ref{universalpropertyclassifyingspace}, which was proven by Dold in \cite{dold}, as Theorem 7.5.
\begin{thm}[Strong universal property of the classifying space]
\label{generalizeduniversalpropertyclassifyingspace}
Let $P$ be a $G$-bundle over a CW-complex $Y$. For each $K\subset Y$ such that $(X,Y)$ is a CW-pair and $\rho:P|_K\rightarrow \tilde EG$ right $G$-equivariant, there exists a right $G$-equivariant morphism
$$\phi:P\longrightarrow\tilde EG$$
that extends $\rho$ to $P$. In particular, the following diagram commutes,
where $r$ and $f$ are the maps induced by $\rho$ and $\phi$ on orbit spaces.
$$\xymatrix{&P|_K\ar[rr]^-{\pi|_K}\ar'[d]^(.5){\rho}[dd]\ar@{^{(}->}[ld]&&K\ar[dd]^-r\ar@{^{(}->}[ld]\\
P\ar[rr]^(.7){\pi}\ar@{-->}[dd]^(.7){\phi}&&Y\ar@{-->}[dd]^(.3){f}&\\
&\tilde EG\ar[rr]|-{\phantom A}&&\tilde BG\\
\tilde EG\ar[rr]\ar@{=}[ru]&&\tilde BG\ar@{=}[ru]&}$$
\end{thm}
In the proof we use the following fact.
\begin{prop}{\cite[Proposition 6.1]{mitchell}}.
\label{equivariantmap}
Let $Q$ be a $G$-bundle over $Y$ and $Z$ a right $G$-space. The tensor product $Q\otimes_GZ$ is identified with the orbit space of the right $G$-action on $Q\times Z$ given by $(q,z)\cdot g:=(q\cdot g,z\cdot g)$. There is a bijection
$$\mathcal Mod_G(Q,Z)\cong\Gamma(Y,Q\otimes_GZ)$$
between the set of right $G$-equivariant maps $Q\rightarrow Z$ and the set of sections of the fibration $Q\otimes_GZ\rightarrow Y$, induced by $\pr_1:Q\times Z\rightarrow Q$. The assignment is given by $\phi\mapsto\sigma_{\phi}:Y\rightarrow Q\otimes_GY$, where $\sigma_{\phi}$ is induced by $(Q,\phi):Q\rightarrow Q\times Z$.
\hfill\ensuremath{\Box}
\end{prop}
\begin{proof}[Proof of Theorem \ref{generalizeduniversalpropertyclassifyingspace}]
According to Proposition \ref{equivariantmap}, the $G$-equivariant morphism $\rho:P|_K\rightarrow\tilde EG$ determines a section $\sigma_{\rho}$ of $P|_K\otimes_G\tilde EG\longrightarrow K$.
This allows us to draw a commutative square.
$$\xymatrix{&P|_K\otimes_G\tilde EG\ar[rd]|-{\iota\otimes_G\tilde EG}&\\
K\ar@/^1pc/[rr]\ar[ru]^-{\sigma_{\rho}}\ar@{^{(}->}[d]_-{i}&&P\otimes_G\tilde EG\ar@{->>}[d]^-{\sim}\\
Y\ar@{=}[rr]&&Y}$$
The vertical arrow on the left is a Serre cofibration. As for the vertical arrow on the right, it is an acyclic Serre fibration. Indeed, $P\otimes_G\tilde EG$ is a fibre bundle over a paracompact space $Y$, every fibre bundle over a paracompact space is \emph{numerable} (\cite[Section 7]{dold}), and every numerable fibre bundle is a Serre fibration (\cite[Theorem 7.12]{spanier}). Moreover, the fibre is contractible, and a Serre fibration with a contractible fibre is a weak equivalence.

Thus there exists a lift $\sigma:Y\rightarrow P\otimes_G\tilde EG$:
$$\xymatrix{&P|_K\otimes_GEG\ar[rd]|-{\iota\otimes_GEG}&\\
K\ar@/^1pc/[rr]\ar[ru]^-{\sigma_{\rho}}\ar@{^{(}->}[d]_-{i}&&P\otimes_GEG\ar@{->>}[d]^-{\sim}\\
Y\ar@{=}[rr]\ar@{-->}[rru]^-{\sigma}&&Y}.$$
From the commutativity of the lower triangle we see that the lift $\sigma$ is a section of $P\otimes_G\tilde EG\rightarrow Y$, and therefore corresponds, according to Proposition \ref{equivariantmap}, to a right $G$-equivariant map $\phi:P\rightarrow\tilde EG$ that induces some $f:Y\rightarrow BG$.

Moreover, using the commutativity of the upper part of the diagram,
we have:
$$(\iota\otimes_GEG)\circ\sigma_{\rho}=\sigma\circ i=(\iota\otimes_GEG)\circ\sigma_{\phi\circ\iota},$$
where the second equality is induced at the level of orbits by the displayed commutative diagram of right $G$-equivariant maps
$$\xymatrix{P|_K\ar[d]_-{\iota}\ar[rr]^-{(P|_K,\phi\circ\iota)}&&P|_K\times Q\ar[d]^-{\iota\times Q}&\quad\ar@{}[d]|-{\leadsto}&K\ar[d]_-{i}\ar[rr]^-{\sigma_{\phi\circ\iota}}&&P|_K\otimes_GQ\ar[d]^-{\iota\otimes_GQ}\\
P\ar[rr]_-{(P,\phi)}&&P\times Q&\quad&Y\ar[rr]_-{\sigma_\phi}&&P\otimes_G Q.}$$

Since $\iota\otimes_GEG$ is injective, the sections are equal, that is $\sigma_{\rho}=\sigma_{\phi\circ\iota}$. Thus, using again Proposition \ref{equivariantmap}, we have
$$\rho=\phi\circ\iota\mbox{ and }r=f\circ i.$$
\end{proof}
We use the result above to conclude that the process of pulling back the universal bundle is injective, in the sense of the following proposition.
\begin{prop}
\label{injectivityofpullback}
If the pullbacks of the universal bundle $\tilde EG$ via two maps $f,g:X\rightarrow BG$ are equivalent $G$-bundles over the CW-complex $X$, then $f$ is homotopic to $g$, i.e.,
$$f^{\star}\tilde EG\simeq g^{\star}\tilde EG\quad\Longrightarrow\quad f\simeq g:X\rightarrow\tilde BG.$$
\end{prop}
\begin{proof}
Here is essentially the same argument used by Dold (Theorem 7.5) and Steenrod (Theorem 19.3), but presented as an application of Theorem \ref{generalizeduniversalpropertyclassifyingspace}. Let $\alpha:f^{\star}\tilde EG\rightarrow g^{\star}\tilde EG$ be an equivalence of $G$-bundles over $X$.
We define a morphism
$$\varphi:\mathbb 2\times(f^{\star}\tilde EG):=f^{\star}\tilde EG\sqcup f ^{\star}\tilde EG\longrightarrow\tilde EG$$
determined by the summands
$$\epsilon_f:f^{\star}\tilde EG\longrightarrow EG\mbox{ and }\epsilon_g\circ\alpha:f^{\star}\tilde EG\longrightarrow g^{\star}\tilde EG\longrightarrow\tilde EG.$$
It is $G$-equivariant because each summand is, and it induces $f+g:\mathbb 2\times X=X\sqcup X\rightarrow\tilde BG$. The inclusion
$$i\times(f^{\star}P):\mathbb 2\times(f^{\star}P)\longrightarrow\mathbb I\times(f^{\star}P)$$
is also right $G$-equivariant.

All the conditions to apply Theorem \ref{generalizeduniversalpropertyclassifyingspace} (with $P=\mathbb I\times(f^{\star}\tilde EG)$, $Y=\mathbb I\times X$, $K=\mathbb 2\times X$) are satisfied. We thus get a morphism of $G$-bundles $\phi:\mathbb I\times(f^{\star}\tilde EG)\rightarrow\tilde EG$ and a map $F:\mathbb I\times X\rightarrow\tilde BG$.
Such an $F$ is the required homotopy.
\end{proof}

As an immediate consequence of Theorem \ref{generalizeduniversalpropertyclassifyingspace}, taking $K$ to be empty, and Proposition \ref{injectivityofpullback}, we obtain the following classical statements.
\begin{prop}[Universal property of the classifying space]
\label{universalpropertyclassifyingspace}
Let $P$ be a $G$-bundle over a CW-complex $X$. There exists a right $G$-equivariant morphism
$$\phi:P\longrightarrow\tilde EG$$
that induces a map on the base spaces
$$f:X\longrightarrow\tilde BG.$$
As a consequence,
$$P\simeq f^{\star}\tilde EG.$$\hfill\ensuremath{\Box}
\end{prop}

\begin{thm}[Classification of $G$-bundles]
\label{classificationGbundles}
Let $X$ be a CW-complex. The pullback of the classifying bundle $\tilde EG$ induces a bijection
$$(-)^{\star}\tilde EG:\mathcal Top(X,\tilde BG)/_{\simeq}\longrightarrow{{}_X\mathcal Bun_G}/_{\simeq}$$
between the homotopy classes of continuous maps $X\rightarrow\tilde BG$ and the equivalence classes of $G$-bundles over $X$.
\hfill\ensuremath{\Box}
\end{thm}

By taking $K$ to be the base point of $X$ in Proposition \ref{generalizeduniversalpropertyclassifyingspace}, we can establish based variants of Proposition \ref{universalpropertyclassifyingspace} and Theorem \ref{classificationGbundles}.

\begin{prop}
\label{pointeduniversalpropertyclassifyingspace}
Let $P$ be a $G$-bundle over a
CW-complex $X$ pointed on a vertex.
There exists a right $G$-equivariant morphism
$$\phi:P\longrightarrow\tilde EG$$
that induces a \emph{based} map at the level of base spaces
$$f:X\longrightarrow\tilde BG.$$
As a consequence,
$$P\simeq f^{\star}\tilde EG.$$
\hfill\ensuremath{\Box}
\end{prop}

\begin{thm}[\emph{Pointed} classifications of $G$-bundles]
\label{pointedclassificationGbundles}
Let $X$ be a CW-complex based at a vertex. The pullback of the universal bundle $\tilde EG$ induces a bijection
$$(-)^{\star}\tilde EG:\mathcal Top_{\star}(X,\tilde BG)/_{\simeq}\longrightarrow{{}_X\mathcal Bun_G}/_{\simeq}$$
between the (unpointed!) homotopy classes of pointed continuous maps $X\rightarrow\tilde BG$ and the equivalence classes of $G$-bundles over $X$.
\hfill\ensuremath{\Box}
\end{thm}
The next result is essentially classical.
\begin{prop}[Uniqueness of the classifying space]
\label{uniquenessclassifyingspace}
The classifying space of a topological group $G$ is unique up to homotopy. More precisely, for a CW-complex $B$ pointed on a vertex
the following are equivalent.
\begin{itemize}
	\item[(1)] \emph{Homotopy type}: $B$ has the same homotopy type as $\tilde BG$.
	\item[(2)] \emph{Universal property}: there exists a natural bijection
$$\mathcal Top(X,B)/_{\simeq}\longrightarrow{{}_X\mathcal Bun_G}/_{\simeq}$$
between the homotopy classes of continuous maps $X\rightarrow B$ and the equivalence classes of $G$-bundles over $X$, for every CW-complex $X$.
	\item[(2')] \emph{Pointed universal property}: there exists a natural bijection
	$$\mathcal Top_{\star}(X,B)/_{\simeq}\longrightarrow{{}_X\mathcal Bun_G}/_{\simeq}$$
between the (unpointed!) homotopy classes of pointed continuous maps $X\rightarrow B$ and the equivalence classes of $G$-bundles over $X$, for every CW-complex $X$ based at a vertex.
	\item[(3)] \emph{Contractibility}: $B$ is a classifying space for $G$, i.e., there exists a $G$-bundle $E$ over $B$ that is weakly contractible.
\end{itemize}
\end{prop}
\begin{proof}
\begin{itemize}
\item $[(3)\Longrightarrow(2')]$: See Proposition \ref{pointeduniversalpropertyclassifyingspace}.
\item $[(2')\Longrightarrow(2)]$: By the Yoneda Lemma, a natural bijection as in $(2')$ has to be induced by pulling back a $G$-bundle $E$ over $B$. Consider the function
$$(-)^{\star}(E):\mathcal Top(X,B)/{\simeq}\longrightarrow{{}_X\mathcal Bun_G}/{\simeq}.$$
This map is injective, by Proposition \ref{injectivityofpullback}, and surjective since its restriction to the subset $\mathcal Top_{\star}/_{\simeq}$ is.
\item $[(2)\Longrightarrow(1)]$: It is an application of the Yoneda Lemma. Indeed, $B$ represents the functor $X\mapsto_X\mathcal{B}un_G/{\simeq}$ on the homotopy category of connected spaces that admit the structure of countable CW-complex.
\item $[(1)\Longrightarrow(3)]$: Let $f:B\rightarrow \tilde BG$ be a homotopy equivalence, and consider the pullback $E:=f^{\star}\tilde EG$ of Milnor's model. It comes with the counit $\epsilon_f:E=f^{\star}\tilde EG\rightarrow\tilde EG$, which is a morphism of $G$-bundles.
There is a morphism of long exact sequences with identities on the component of the fibres and isomorphisms on the components of the base space. It follows that $E$ is weakly equivalent to $\tilde EG$, which is contractible.
\end{itemize}
\end{proof}
\section{Classification of principal bundles with fixed base space}
For a nice pointed space $X$, we now provide a classification of principal bundles over $X$, stated as Theorem \ref{classificationbundlesoverX}. Milnor described in \cite{milnor1} a loop space model $\tilde\Omega$ that is in fact a group and plays a role dual to that of the classifying space. The loop space $\tilde\Omega X$ is the structure group of a universal bundle over $X$ that generates all the others by \emph{pushing forward} along some continuous group homomorphisms. Our contribution is a description of when two such continuous homomorphisms induce the same bundle. We also define a reasonable notion of $\tilde{\Omega}f$, for a pointed map $f$, so that $\tilde\Omega$ is pseudofunctorial, and the classification is natural with respect to $X$.

In Theorem \ref{adjunction} we put together the two sides of the coin to produce an adjunction in a homotopical context between the classifying space functor and Milnor's loop space. This adjunction is homotopically well-behaved, in the sense described by Farjoun and Hess in \cite{fh}.

Although the structure of this section is at the beginning as similar as possible to the previous one, the logic of the arguments is quite different. Indeed, many of the proofs depend on results from the previous section. In particular, the following proposition allows some interaction between the context of the previous section and this section.
\begin{prop}
\label{naturality1}
Given a morphism of topological groups $a:G\rightarrow H$,
there is an equivalence of $H$-bundles over $\tilde BG$
$$a_{\star}\tilde EG\simeq(\tilde Ba)^{\star}\tilde EH.$$
\end{prop}
\begin{proof}
The projection of $\tilde EG$ onto $\tilde BG$ and the map $\tilde Ea$ induce a map $\phi:\tilde EG\rightarrow\tilde Ba^{\star}\tilde EH$, as indicated by the following diagram.
$$\xymatrix{\tilde EG\ar@/^1pc/[rrd]^-{\pi}\ar@/_1pc/[rdd]_-{\tilde Ea}\ar@{-->}[rd]|-{\phi}&&\\
&\tilde Ba^{\star}\tilde EH\pb\ar[r]|-{\pi^{\star}}\ar[d]|-{\epsilon_{\tilde Ba}}&\tilde BG\ar[d]|-{\tilde Ba}\\
&\tilde EH\ar[r]&\tilde BH}$$
By the commutativity of the upper triangle, it follows that $\phi$ respects the projection over $\tilde BG$. Also, one can check that $\phi$ induces $a$ on each fibre, and it is therefore $a$-equivariant. As a consequence we obtain the desired equivalence.
\end{proof}

Mimicking the approach we took in the dual case, we would like to describe the classification problem for principal bundles over a fixed base space $X$ in terms of the following map:
$$(-)_{\star}Q:\xymatrix{\mathcal Gp(\mathcal Top)(G,H)\ar[r]&_X\mathcal Bun_G\ar@{->>}[r]&_X\mathcal Bun_H/\simeq},$$
where $Q$ is a $G$-bundle over $X$ and $[a:G\rightarrow H]\mapsto a_{\star}Q$. 
Unlike the dual case, this does not induce an interesting assignment on a quotient of $\mathcal Gp(\mathcal Top)(G,H)$.
We will see that the kernel is computable when $Q=\tilde PX$ and is given by the following equivalence relation.
\begin{dfnt} Two continuous homomorphisms $a,b:G\rightarrow H$ are said to be \textbf{algebraically equivalent} if $\tilde Ba$ is homotopic to $\tilde Bb$, which we denote $a\equiv b$. 
\end{dfnt}
Algebraic equivalence of continuous group homomorphisms is easy to define, but hard to imagine. The feeling is that it should be somehow related to a notion of homotopy that respects the algebraic structure. A good candidate is the following.

Given a topological group $H$, the space $H^{\mathbb I}$ of paths in $H$ is a topological group with respect to the pointwise multiplication. 
\begin{dfnt} Two continuous homomorphisms $a,b:G\rightarrow H$ are said to be \textbf{algebraically homotopic} if $a$ is homotopic to $b$ via an \textbf{algebraic homotopy}, i.e., a homotopy $F:G\rightarrow H^{\mathbb I}$ that is also a homomorphism of topological groups.
\end{dfnt}
\begin{prop} If two continuous homomorphisms $a,b:G'\rightarrow G$ are algebraically homotopic, then they are algebraically equivalent.
\end{prop}
\begin{proof} Consider the continuous map
$$\tilde E(G^{\mathbb I})\times\mathbb I\longrightarrow\tilde  EG$$
given by $(\sum_{i\in\mathbb N}t_i\cdot\alpha_i,t)\mapsto\sum_{i\in\mathbb N}t_i\cdot\alpha_i(t)$. This induces a map
$$\tilde B(G^{\mathbb I})\times\mathbb I\longrightarrow\tilde BG,$$
whose adjoint is denoted
$$\rho_G:\tilde B(G^{\mathbb I})\longrightarrow(\tilde BG)^{\mathbb I}.$$
Now, if $F:G'\rightarrow G^{\mathbb I}$ is an algebraic homotopy between $a$ and $b$, then
$$\rho_G\circ\tilde BF:\xymatrix{\tilde BG'\ar[r]^-{\tilde BF}&\tilde B(G^{\mathbb I})\ar[r]^-{\rho_G}&(\tilde BG)^{\mathbb I}}$$
is a homotopy between $\tilde Ba$ and $\tilde Bb$. Therefore $a$ and $b$ are algebraically equivalent.
\end{proof}
\begin{rem}[Naturality of the classification of bundles with a fixed structure group]
\label{naturalitybijection1}
As a consequence of Proposition \ref{naturality1}, the bijections described in Theorems \ref{classificationGbundles} and \ref{pointedclassificationGbundles} are \textbf{natural} in both the variables, with respect to (based) continuous maps up to homotopy and continuous homomorphisms up to algebraic equivalence. In fact, the following diagrams commute up to equivalence.
$$\xymatrix{X\ar[d]_-f&\mathcal Top_{(\star)}(X,\tilde BG)\ar[rr]^-{(-)^{\star}\tilde EG}&&_X\mathcal Bun_G\\
Y&\mathcal Top_{(\star)}(Y,\tilde BG)\ar[u]^-{-\circ f}\ar[rr]_-{(-)^{\star}\tilde EG}&&_Y\mathcal Bun_G\ar[u]_-{f^{\star}(-)}}$$
$$\xymatrix{G\ar[d]_-a&\mathcal Top_{(\star)}(X,\tilde BG)\ar[d]_-{\tilde Ba\circ-}\ar[rr]^-{(-)^{\star}\tilde EG}&&_X\mathcal Bun_G\ar[d]^-{a_{\star}(-)}\\
H&\mathcal Top_{(\star)}(X,\tilde BH)\ar[rr]_-{(-)^{\star}\tilde EH}&&_X\mathcal Bun_H}$$
\end{rem}
The first attempt to define a \emph{universal bundle} for a base space $K$ was for $K$ a connected simplicial complex, pointed on one of its vertices. Milnor defines a group $\tilde\Omega K$, and a $\tilde\Omega K$-bundle over $K$ that are respectively a strictification of the loop space $\Omega K$ and the based path space $PK$. Indeed, using the cellular structure of $K$, Milnor selected a class of easy paths and loops, avoiding any issue related to the parametrization.

Let $K$ be a connected simplicial complex, with a fixed vertex $\overline x$,
\begin{itemize}
\item[>>] We denote by $L_nK$ the set of $n+1$-uples $(x_n,\dots,x_0)$ of points of $K$ such that two consecutive elements lie in a common simplex of $K$, topologized as a subspace of $K^{n+1}$.
\item[>>] We denote by $\tilde LK$ the set of \textbf{Milnor free paths}
$$\tilde LK:=({\bigcup_{n\in\mathbb N}L_nK})/_{\sim},$$
where the relation is generated by the conditions $(x_n,\dots,x_i,\dots,x_0)\sim(x_n,\dots,\hat x_i,\dots,x_0)$, when either $x_i=x_{i-1}$ or $x_{i-1}=x_{i+1}$.
\item[>>] We denote by $\tilde PK$ the set of \textbf{based Milnor paths}, i.e., is the subset of $\tilde LK$, whose elements are of the form $[(x_n,x_{n-1},\dots,x_1,\overline x)]$.
\item[>>] We denote by $\tilde\Omega K$ the set of \textbf{based Milnor loops}, i.e., the subset of $\tilde LK$, whose elements are of the form $[(\overline x,x_{n-1},\dots,x_1,\overline x)]$.
\end{itemize}
\begin{thm}\cite[Sections 3,5]{milnor1}
\label{classifyinggroup}
Let $K$ be a connected simplicial complex with a fixed vertex. The space $\tilde PK$ is contractible and has the structure of a $\tilde\Omega K$-bundle over $X$. Moreover, $\tilde\Omega K$ is a is a nice topological group. Furthermore, given a $G$-bundle $P$ over $K$, there exists a morphism of bundles over $K$,
$$\phi:\tilde PK\longrightarrow P,$$
that is equivariant with respect to a homomorphism induced between the fibres
$$a:\tilde\Omega K\longrightarrow G.$$
As a consequence,
$$a_{\star}\tilde PK\simeq P.$$
\end{thm}
The constructions of $\tilde P$ and $\tilde\Omega$ for simplicial complexes can be extended to connected countable CW-complexes, based at vertices, producing a bundle with the same properties. The idea is to first replace such a space by a homotopy equivalent simplicial complex.

Let $X$ be a connected countable CW-complex, based at one of its vertices. As we show in \cite[Theorem 1.3]{whitehead}, there exists a locally finite simplicial complex $K$ and a homotopy equivalence $s:X\rightarrow K$. Up to further subdivision of $K$, which does not change the topology because $K$ is locally finite, we can assume that the image of the base point through $s$ is a vertex of $K$. With a global definition of $\tilde P$ and $\tilde\Omega$ on $\mathcal T_{\star}$ in mind, it is convenient to choose such a simplicial replacement for every space in $\mathcal T_{\star}$.
\begin{choice}
Let $X\in\mathcal T_{\star}$.
We choose
\begin{itemize}
\item a countable CW-structure such that the base point is a vertex, and
\item a homotopy equivalence $s_X:X\rightarrow K(X)$, where $K(X)$ is a locally finite simplicial complex, and the image of the base point through $s_X$ is a vertex of $K(X)$. When $X$ admits a simplicial structure such that the base point is a vertex we take $s_X$ to be the identity.
\end{itemize}
\end{choice}
By setting
$$\tilde\Omega X:=\tilde\Omega(K(X))\text{ and }\tilde PX:=s_X^{\star}\tilde PK(X),$$
we obtain the desired bundle.
\begin{prop}\cite[Corollary 3.7]{milnor1}
\label{classifyinggroupforCWcomplexes}
\label{universalpropertyclassifyinggroup}
Let $X$ be a connected countable CW-complex, with a fixed vertex. The space $\tilde PX$ is contractible, and has the structure of a $\tilde\Omega X$-bundle over $X$. Moreover, $\tilde\Omega X$ is a nice topological group. Futhermore, given a a $G$-bundle $P$ over $X$ there exists a morphism of bundles over $X$,
$$\phi:\tilde PX\longrightarrow P,$$
that is equivariant with respect to a homomorphism induced between the fibres
$$a:\tilde\Omega X\longrightarrow G.$$
As a consequence,
$$a_{\star}\tilde PX\simeq P.$$
\end{prop}
The last part of the statement is dual to Proposition \ref{universalpropertyclassifyingspace}, and shows that $\tilde PX$ can be thought of as a \textbf{universal bundle} over $X$. It is important to note that this universal property is not valid (a priori) for other contractible bundles over $X$; compare with Proposition \ref{uniquenessclassifyinggroup}.

Given a space $X\in\mathcal T_{\star}$, since $\tilde PX$ is a $\tilde{\Omega}X$-bundle over $X$, as a consequence of Theorem \ref{universalpropertyclassifyingspace}, there exists a morphism $\tilde PX\rightarrow\tilde E\tilde{\Omega}X$ of $\tilde{\Omega}X$-bundles inducing a map on the orbit spaces $X\rightarrow B\tilde{\Omega}X$, whose pullback of $\tilde E\tilde{\Omega}X$ is equivalent to $\tilde PX$.
\begin{choice}[Unit]
\label{unit}
Let $X\in\mathcal T_{\star}$. We choose a morphism of $\tilde{\Omega}X$-bundles
$$\tilde PX\longrightarrow\tilde E\tilde{\Omega}X,$$
and denote the induced weak equivalence
$$\eta_X:X\longrightarrow B\tilde{\Omega}X,$$
which induces an equivalence of $\tilde\Omega X$-bundles over $X$
$$\eta_X^{\star}(E\tilde{\Omega}X)\simeq \tilde PX.$$
\end{choice}
The collection $\eta$ that we just chose will play the role of a \textbf{unit} of the hypothetical adjunction. Once proven that $\eta$ is \emph{homotopy universal}, the analogues of the other properties that characterize an adjunction will follow formally.
\begin{prop}
[Universality of the unit]
\label{universalityunit}
For every $X\in\mathcal T_{\star}$, the unit $\eta_X$ is \textbf{homotopy universal} among the continuous maps from $X$ into a classifying space $\tilde B$. In other words, for any pointed continuous map $f:X\rightarrow\tilde BG$ there exists a unique (up to algebraic equivalence) homomorphism of topological groups $a:\tilde{\Omega}X\rightarrow G$ such that $Ba\circ\eta_X\simeq f$.
$$\xymatrix{X\ar[r]^-{\eta_X}\ar[rd]_-f&B\tilde{\Omega}X\ar[d]^-{Ba}&\tilde{\Omega}X\ar@{-->}[d]^-a\\
&BG&G}$$
\end{prop}
\begin{proof}
The pullback $f^{\star}\tilde EG$ is a $G$-bundle over $X$. Thus, by Proposition \ref{universalpropertyclassifyinggroup}, there exists $a:\tilde{\Omega}X\longrightarrow G$ such that $a_{\star}\tilde PX\simeq f^{\star}\tilde EG$.
As a consequence of Proposition \ref{naturality1} and Choice \ref{unit}, we have the following equivalence of bundles:
$$(\tilde Ba\circ\eta_X)^{\star}\tilde EG\simeq\eta_X^{\star}(\tilde Ba)^{\star}\tilde EG\simeq\eta_X^{\star}(a_{\star}(\tilde E\tilde{\Omega}X))\simeq a_{\star}(\eta_X^{\star}(\tilde E\tilde{\Omega}X))\simeq a_{\star}\tilde PX\simeq f^{\star}\tilde EG.$$
Thus, by Proposition \ref{injectivityofpullback}, $\tilde Ba\circ\eta_X\simeq f$.

For the uniqueness, suppose $a':\tilde{\Omega}X\rightarrow G$ also satisfies the condition. Then, if $\eta^{-1}$ denotes a homotopy inverse for $\eta_X$, we have that
$$\tilde Ba\simeq\tilde Ba\circ\id{}_{\tilde{\Omega}X}\simeq\tilde Ba\circ\eta_X\circ\eta^{-1}\simeq f\circ\eta^{-1}\simeq\tilde Ba'\circ\eta_X\circ\eta^{-1}\simeq\tilde Ba'\circ\id{}_{\tilde{\Omega}X}\simeq\tilde Ba'.$$
So $a\equiv a'$.
\end{proof}
Using the unit $\eta_X$ we are finally able to describe when two group homomorphisms defined on $\tilde\Omega X$ induce the same bundle over $X$.
\begin{prop}
\label{injectivityofpushforward}
Let $X\in\mathcal T_{\star}$. Two topological group morphisms $\tilde\Omega X\rightarrow G$ are algebraically equivalent if and only if their pushforwards of $\tilde PX$ are equivalent bundles of $G$-bundles over $X$, i.e.,
$$a\equiv b:\tilde\Omega X\rightarrow G\quad\Longleftrightarrow\quad a_{\star}\tilde PX\simeq b_{\star}\tilde PX.$$
\end{prop}
\begin{proof}
As a consequence of Proposition \ref{naturality1} and Choice \ref{unit}, we have the following equivalences of $G$-bundles over $X$,
$$(\tilde Ba\circ\eta_X)^{\star}\tilde EG
\simeq\eta_X^{\star}(\tilde Ba)^{\star}\tilde EG
\simeq\eta_X^{\star}a_{\star}\tilde E(\tilde{\Omega}X)
\simeq a_{\star}\eta_X^{\star}\tilde E(\tilde{\Omega}X)
\simeq a_{\star}\tilde PX,$$
and similarly $$(\tilde Bb\circ\eta_X)^{\star}\tilde EG\simeq b_{\star}\tilde PX.$$
Thus, using Proposition \ref{homotopicmapsgiveequivalentbundles},
$$a\equiv b\quad\Longleftrightarrow\quad Ba\simeq Bb\quad\Longleftrightarrow\quad Ba\circ\eta_X\simeq Bb\circ\eta_X\quad\Longleftrightarrow$$
$$\Longleftrightarrow\quad(Ba\circ\eta_X)^{\star}\tilde E(\tilde\Omega X)\simeq(Bb\circ\eta_X)^{\star}\tilde E(\tilde\Omega X)\quad\Longleftrightarrow\quad a_{\star}\tilde PX\simeq b_{\star}\tilde PX.$$
\end{proof}
Proposition \ref{universalpropertyclassifyinggroup} and Proposition \ref{injectivityofpushforward} together imply the following.
\begin{thm}[Classification of bundles over $X$]
\label{classificationbundlesoverX}
Let $X\in\mathcal T_{\star}$.\\
The map $(-)_{\star}\tilde PX$ induces a bijection
$$(-)_{\star}\tilde PX:\mathcal Gp(\mathcal Top)(\tilde\Omega X,G)/_{\equiv}\longrightarrow{{}_X\mathcal Bun_G}/_{\simeq}$$
between the equivalence classes of continuous homomorphisms $\tilde\Omega X\rightarrow G$ and the equivalence classes of $G$-bundles over $X$.\qedhere
\end{thm}
The next aim is to make $\tilde{\Omega}$ act on continuous maps, so that the classification described in Theorem \ref{classificationbundlesoverX} is \emph{natural}. Unfortunately the nature of the values of $\tilde{\Omega}$ on objects involves arbitrary choices concerning the cell decomposition, and there is no canonical way to assign a homomorphism to a continuous function. Instead we require that $\tilde{\Omega}f$ be consistent with $f$ in terms of the bundle they produce, imitating Proposition \ref{naturality1}.

Let $f:X\rightarrow Y$ be a continuous map in $\mathcal T_{\star}$. Then $f^{\star}\tilde PY$ is a $\tilde{\Omega}Y$-bundle over $X$. As a consequence of Proposition \ref{universalpropertyclassifyinggroup}, there exists a morphism $\tilde PX\rightarrow f^{\star}\tilde PY$ of bundles over $X$ inducing a homomorphism on the fibres $\tilde{\Omega}X\rightarrow\tilde{\Omega}Y$ whose pushforward of $\tilde PX$ is equivalent to $f^{\star}\tilde PY$.
\begin{choice}
\label{choiceOmegaonarrows}
Let $f:X\rightarrow Y$ be a continuous map in $\mathcal T_{\star}$. We choose a morphism of bundles over $X$
$$\phi_f:\tilde PX\longrightarrow f^{\star}\tilde PY,$$ and denote the restriction to the fibre
$$\tilde{\Omega}f:\tilde{\Omega}X\longrightarrow\tilde{\Omega}Y,$$
which induces an equivalence
$${\phi_f}_{\star}:f^{\star}\tilde PY\simeq(\tilde{\Omega}f)_{\star}(PX).$$
By Proposition \ref{injectivityofpushforward}, the choice of $\tilde{\Omega}f$ is determined up to algebraic equivalence.
\end{choice}
\begin{prop}[Functoriality of the universal bundle over a fixed space]
\label{functorialityclassifyinggroup}
The classifying group $\tilde\Omega$ gives a pseudofunctorial assignment
$$\tilde{\Omega}:\mathcal T_{\star}\longrightarrow\mathcal G,$$
i.e., the following hold:
\begin{itemize}
\item \emph{Identity}: $\tilde{\Omega}\id_X\equiv\id_{\tilde{\Omega}X}$, for any $X\in\mathcal T_{\star}$;
\item \emph{Composition}: $\tilde{\Omega}(g\circ f)\equiv(\tilde{\Omega}g)\circ(\tilde{\Omega}f)$, for any $f:X\rightarrow Y$ and $g:Y\rightarrow Z$ in $\mathcal T_{\star}$.
\end{itemize}
\end{prop}
\begin{proof}
The assignment on objects is well defined, thanks to Proposition \ref{classifyinggroupforCWcomplexes}. As for the functorial properties, use Choice \ref{choiceOmegaonarrows} to get
the following equivalences of bundles:
$$(\tilde{\Omega}(g\circ f))_{\star}\tilde PX\simeq(g\circ f)^{\star}\tilde{P}Z\simeq
f^{\star}g^{\star}\tilde{P}Z\simeq f^{\star}(\tilde{\Omega}g)_{\star}\tilde{P}Y\simeq$$
$$\simeq(\tilde{\Omega}g)_{\star}f^{\star}\tilde{P}Y\simeq(\tilde{\Omega}g)_{\star}(\tilde{\Omega}f)_{\star}\tilde{P}X=(\tilde{\Omega}g\circ\tilde{\Omega}f)_{\star}\tilde PX.$$
Thus, by Proposition \ref{injectivityofpushforward}, $\tilde{\Omega}(g\circ f)\simeq\tilde{\Omega}g\circ\tilde{\Omega}f$.
The argument for the unit is analogous.
\end{proof}
\begin{rem}[Naturality of the classification of bundles over a fixed base space]
\label{naturalitybijection2}
The bijection described in Theorem \ref{classificationbundlesoverX} is \textbf{natural} in both the variables, with respect to pointed continuous maps and continuous homomorphisms. Indeed, the following diagrams commute up to equivalence.
$$\xymatrix{X\ar[d]_-f&\mathcal Gp(\mathcal Top)(\tilde{\Omega}X,G)\ar[rr]^-{(-)_{\star}\tilde PX}&&_X\mathcal Bun_G\\
Y&\mathcal Gp(\mathcal Top)(\tilde{\Omega}Y,G)\ar[u]^-{-\circ\tilde{\Omega}f}\ar[rr]_-{(-)_{\star}\tilde PY}&&_Y\mathcal Bun_G\ar[u]_-{f^{\star}(-)}}$$
$$\xymatrix{G\ar[d]_-a&\mathcal Gp(\mathcal Top)(\tilde{\Omega}X,G)\ar[d]_-{a\circ-}\ar[rr]^-{(-)_{\star}\tilde PX}&&_X\mathcal Bun_G\ar[d]^-{a_{\star}(-)}\\
H&\mathcal Gp(\mathcal Top)(\tilde{\Omega}X,H)\ar[rr]_-{(-)_{\star}\tilde PX}&&_X\mathcal Bun_H}$$
\end{rem}
\begin{thm}[Looping-Delooping adjunction]
\label{adjunction}
For every $X\in\mathcal T_{\star}$ and $G\in\mathcal G$ there is a bijection
$$\mathcal G(\tilde\Omega X,G)/{\equiv}\longrightarrow\mathcal T_{\star}(X,\tilde BG)/{\simeq},$$
given by $a\mapsto\tilde Ba\circ\eta_X$. It is natural as explained in Remarks \ref{naturalitybijection1} and \ref{naturalitybijection2}.
In this sense Milnor's loop space pseudofunctor $\tilde\Omega$ and the classifying space functor $\tilde B$ form a \textbf{homotopy adjunction}
$$\tilde{\Omega}:\xymatrix{\mathcal T_{\star}\ar@<1ex>[r]&\mathcal G\ar@<1ex>[l]}:\tilde B.$$
\end{thm}
\begin{proof}
Let $a,b:\tilde{\Omega}X\rightarrow G$ be continuous homomorphisms.
Since $\eta_X$ is a homotopy equivalence, we have the following equivalences
$$a\equiv b\Longleftrightarrow Ba\simeq Bb\Longleftrightarrow Ba\circ\eta_X\simeq Bb\circ\eta_X.$$
This shows that the assignment is well defined and injective. The surjectivity is a consequence of Proposition \ref{universalityunit}.
\end{proof}
\begin{prop}[Naturality of the unit]
\label{naturalityunit}
The unit is a \textbf{homotopy natural transformation}
$$\eta:\Id{}_{\mathcal T_{\star}}\longrightarrow\tilde B\tilde{\Omega},$$
i.e., for any $f:X\rightarrow Y$ in $\mathcal T_{\star}$, the following square commutes up to homotopy.
$$\xymatrix{X\ar[r]^-{\eta_X}\ar[d]_-{f}&\tilde B\tilde{\Omega}X\ar[d]^-{\tilde B\tilde{\Omega}f}\\
Y\ar[r]_-{\eta_Y}&\tilde B\tilde{\Omega}Y}$$
\end{prop}
\begin{proof}
Using Proposition \ref{naturality1}, Choice \ref{unit} and Choice \ref{choiceOmegaonarrows},
we have the following equivalences of bundles:
$$(\tilde B\tilde{\Omega}f\circ\eta_X)^{\star}\tilde E(\tilde{\Omega}Y)\simeq\eta_X^{\star}(\tilde B\tilde{\Omega}f)^{\star}\tilde E(\tilde{\Omega}Y)\simeq\eta_X^{\star}(\tilde{\Omega}f)_{\star}\tilde E(\tilde{\Omega}X)\simeq$$
$$\simeq(\tilde{\Omega}f)_{\star}\eta_X^{\star}\tilde E(\tilde{\Omega}X)\simeq(\tilde{\Omega}f)_{\star}\tilde{P}X\simeq f^{\star}\tilde{P}Y\simeq f^{\star}\eta_Y^{\star}\tilde E(\tilde{\Omega}Y)\simeq(\eta_Y\circ f)^{\star}\tilde E(\tilde{\Omega}Y).$$
Thus, by Proposition \ref{injectivityofpullback}, $B\tilde{\Omega}f\circ\eta_X\simeq \eta_Y\circ f$.
\end{proof}
Let $G\in\mathcal G$. Then $\tilde EG$ is a $G$-bundle over $\tilde BG$. As a consequence of Proposition \ref{universalpropertyclassifyinggroup}, there exists a morphism $\tilde P\tilde BG\rightarrow \tilde EG$ of bundles over $\tilde BG$ inducing a homomorphism on the fibre $\tilde{\Omega}\tilde BG\rightarrow G$, whose pushforward of $\tilde P\tilde BG$ is equivalent to $\tilde EG$.
\begin{choice}[Counit]
\label{counit}
Let $G\in\mathcal G$. We choose a morphism of bundles over $\tilde BG$
$$\tilde P\tilde BG\longrightarrow\tilde EG,$$
and denote the restriction to the fibre
$$\epsilon_G:\tilde{\Omega}\tilde BG\longrightarrow G,$$
which induces an equivalence
$${\epsilon_G}_{\star}\tilde P(\tilde BG)\simeq\tilde EG.$$
\end{choice}
The collection $\epsilon$ that we just chose is the \textbf{counit} of the homotopy adjunction.
\begin{prop}[Universality of the counit]
For every $G\in\mathcal G$, the counit $\epsilon_G$ is \textbf{homotopy universal} among the homomorphisms of topological groups from Milnor's loop space $\tilde{\Omega}$ to $G$. In other words, for any continuous homomorphism $a:\tilde{\Omega}X\rightarrow G$ there exists a unique (up to homotopy) pointed continuous map $f:X\rightarrow\tilde BG$ such that $\epsilon_G\circ\tilde{\Omega}f\equiv a$.
$$\xymatrix{\tilde BG&\tilde{\Omega}\tilde BG\ar[r]^-{\epsilon_G}&G\\
X\ar@{-->}[u]^-f&\tilde{\Omega}X\ar[ru]_-a\ar[u]^-{\tilde{\Omega}f}&}$$
\end{prop}
\begin{proof}
The pushforward $a_{\star}\tilde PX$ is a $G$-bundle over $X$. Thus, by Theorem \ref{universalpropertyclassifyingspace}, there exists $f:X\rightarrow\tilde BG$ such that $f^{\star}\tilde EG\simeq a_{\star}\tilde PX$. Using Choice \ref{choiceOmegaonarrows} and Choice \ref{counit},
we have the following equivalences of bundles:
$$(\epsilon_G\circ\tilde{\Omega}f)_{\star}\tilde PX\simeq{\epsilon_G}_{\star}\tilde{\Omega}f_{\star}\tilde PX\simeq {\epsilon_G}_{\star}f^{\star}\tilde P(\tilde BG)\simeq f^{\star}{\epsilon_G}_{\star}\tilde P(\tilde BG)\simeq f^{\star}\tilde EG\simeq a_{\star}\tilde PX.$$
Thus, by Proposition \ref{injectivityofpushforward}, $\epsilon_G\circ\tilde{\Omega}f\simeq a$.

For the uniqueness, we first remark that $\tilde B\epsilon_G\circ\eta_{\tilde BG}\simeq\id_{\tilde BG}$. Indeed,
$$(\tilde B\epsilon_G\circ\eta_{\tilde BG})^{\star}\tilde EG\simeq
(\eta_{\tilde BG})^{\star}(\tilde B\epsilon_G)^{\star}\tilde EG\simeq
(\eta_{\tilde BG})^{\star}(\epsilon_G)_{\star}\tilde E(\tilde{\Omega}\tilde BG)\simeq$$
$$\simeq(\epsilon_G)_{\star}(\eta_{\tilde BG})^{\star}\tilde E(\tilde{\Omega}\tilde BG)\simeq(\epsilon_G)_{\star}\tilde P(\tilde BG)\simeq\tilde EG\simeq\id{}_{\tilde BG}^{\star}\tilde EG.$$
Using this identity and the naturality of the unit, we see that $f$ is determined by $a$ as follows:
$$f=\id{}_{\tilde BG}\circ f\simeq\tilde B\epsilon_G\circ\eta_{\tilde BG}\circ f\simeq\tilde B\epsilon_G\circ\tilde B\tilde{\Omega}f\circ\eta_{X}=\tilde B(\epsilon_G\circ\tilde{\Omega}f)\circ\eta_X\simeq Ba\circ \eta_X,$$
and therefore $f$ is unique up to homotopy.
\end{proof}
\begin{prop}[Naturality of the counit]
\label{naturalitycounit}
The counit is a \textbf{homotopy natural transformation}
$$\epsilon:\tilde{\Omega}\tilde B\longrightarrow\Id{}_{\mathcal G},$$
i.e., for any
$a:G\rightarrow H$ in $\mathcal G$, the following square commutes up to algebraic equivalence.
$$\xymatrix{\tilde{\Omega}\tilde BG\ar[r]^-{\epsilon_G}\ar[d]_-{\tilde{\Omega}\tilde Ba}&G\ar[d]^-a\\\tilde{\Omega}\tilde BH\ar[r]_-{\epsilon_H}&H}$$
\end{prop}
\begin{proof}
Using Choice \ref{choiceOmegaonarrows}, Choice \ref{counit}, and Proposition \ref{naturality1},
we have the following equivalences of bundles:
$$(\epsilon_H\circ\tilde{\Omega}Ba)_{\star}\tilde{P}(\tilde BG)\simeq{\epsilon_H}_{\star}(\tilde{\Omega}Ba)_{\star}\tilde{P}(\tilde BG)\simeq{\epsilon_H}_{\star}(Ba)^{\star}\tilde{P}(\tilde BH)\simeq$$
$$\simeq(\tilde Ba)^{\star}{\epsilon_H}_{\star}\tilde{P}(\tilde BH)\simeq(\tilde Ba)^{\star}\tilde EH\simeq a_{\star}\tilde EG\simeq a_{\star}{\epsilon_G}_{\star}\tilde{P}(\tilde BG)\simeq(a\circ\epsilon_G)_{\star}\tilde{P}(\tilde BG).$$
Thus, by Proposition \ref{injectivityofpushforward}, $\epsilon_H\circ\tilde{\Omega}Ba\equiv a\circ\epsilon_G$.
\end{proof}
Thanks to Proposition \ref{uniquenessclassifyingspace}(2), every contractible $G$-bundle gives a classification of $G$-bundles. Moreover, the specific model that Milnor suggests for the classifying space is not special; equivalent results hold for any model of the classifying space.

In the dual picture, having a contractible bundle over a space does not guarantee (a priori) any universal property among bundles over the same space.
\begin{prop}[Uniqueness of the classifying group]
\label{uniquenessclassifyinggroup}
Let $X\in\mathcal T_{\star}$.\\
For $\Omega\in\mathcal G$ the following are equivalent.
\begin{itemize}
	\item[(1)] \emph{Equivalence}: $\Omega$ is \textbf{equivalent} to $\tilde{\Omega}X$, i.e., there exist continuous homomorphisms
$$\xymatrix{\Omega\ar@<.7ex>[r]^-b&\tilde{\Omega}X\ar@<.7ex>[l]^-a}$$
such that $a\circ b\equiv\id_{\Omega}$ and $b\circ a\equiv\id_{\tilde{\Omega}X}$;
	\item[(2)] \emph{Universal property}: there exists a natural bijection
	$$\mathcal G(\Omega,G)/_{\equiv}\longrightarrow{{}_X\mathcal Bun_G}/_{\simeq}$$
between the equivalence classes of continuous homomorphisms $\Omega\rightarrow G$ and the equivalence classes of $G$-bundles over $X$, for every $G$.
\end{itemize}
Moreover, (1) and (2) imply the following equivalent conditions.
\begin{itemize}
	\item[(3)] \emph{Contractibility}: there exists an $\Omega$-bundle $P$ over $X$ that is contractible;
	\item[(1')] \emph{Milnor's equivalence}:
there exists a nice group $A\in\mathcal G$ and continuous homomorphisms that are homotopy equivalences
$$\xymatrix{\Omega&A\ar[r]^-{\simeq}\ar[l]_-{\simeq}&\tilde{\Omega}X}.$$ 
\end{itemize}
\end{prop}
\begin{proof}
\begin{itemize}
\item $[(1)\Longleftrightarrow(2)]$: It an application  of the Yoneda Lemma. Indeed, $\tilde{\Omega}X$ represents the functor $G\mapsto_X\mathcal{B}un_G/_{\simeq}$ on the category $\mathcal G$ up to equivalence.
\item $[(1')\Longleftrightarrow(3)]$: See \cite[Theorem 5.2.(4)]{milnor2}.
\item $[(1)\Longrightarrow(1')]$: It is enough to take
	$$\xymatrix{\Omega&\tilde{\Omega}X\ar@{=}[r]\ar[l]_-{\simeq}^-a&\tilde{\Omega}X}.$$ 
\end{itemize}
\end{proof}

\bibliography{references}
\bibliographystyle{amsalpha}
\end{document}